\documentclass{article}
\title{Comparing and simplifying distinct-cluster phylogenetic networks}

\author{Stephen J. Willson\\
		Department of Mathematics\\
		Iowa State University\\
		Ames, IA 50011 USA\\
		swillson@iastate.edu}

\usepackage{amsmath}
\usepackage{amsthm}
\usepackage{amssymb}

\usepackage{graphicx}

\usepackage{verbatim} 
\newtheorem{lem}{Lemma}[section]
\newtheorem{thm}[lem]{Theorem}
\newtheorem{cor}[lem]{Corollary}

\begin{document}

\maketitle

\textbf{Abstract:}
Phylogenetic networks are rooted acyclic directed graphs in which the leaves are identified with members of a set $X$ of species.  The cluster of a vertex is the set of leaves that are descendants of the vertex.  A network is ``distinct-cluster'' if  distinct vertices have distinct clusters. This paper focuses on the set $DC(X) $ of distinct-cluster networks whose leaves are identified with the members of  $X$.  For a fixed $X$, a metric on $DC(X)$ is defined.  There is a ``cluster-preserving" simplification process by which vertices or certain arcs may be removed without changing the clusters of any remaining vertices.  Many of the resulting networks may be uniquely determined without regard to the order of the simplifying operations. 

\textbf{Short running title:} Comparing distinct-cluster phylogenetic networks



\section {Introduction} 

It is common in biology to describe evolutionary history by means of a phylogenetic tree $T$.  (See the book \cite{sem03} for many details.) In such a tree, the leaf set corresponds to a set $X$ of taxa on which measurements such as on DNA can be made.  Internal vertices correspond to ancestral species.  Branching corresponds to speciation events by some isolation mechanism.  Typically the trees are assumed to be rooted in the distant past, perhaps by means of an outgroup.  

Recently, the roles of hybridization and lateral gene transfer have been seen to be important; see  \cite{rhy96},  \cite{doo07}, \cite{bou03}, \cite{min13}.  Models of such events sometimes lead to rooted acyclic networks which, unlike trees, allow branches to recombine.  Overviews  for such networks are found in \cite{mor04}, \cite{nak05}, \cite{mor11}, \cite {hus10}.

The number of possible rooted acyclic networks with a given leaf set $X$ is infinite.  Some researchers have focused attention on networks with additional properties.   Such classes include networks that are regular \cite{bar04}, normal \cite{wil10}, tree-child \cite{crv09},  galled trees \cite{gus04}, or level-$k$ \cite{van09}.  

Formal definitions are given in Section 2.  This introduction will give a rough idea of some of the critical concepts.

If $N$ is a finite acyclic rooted network with leaf set $X$, for each vertex $v$ the \emph{cluster} $cl(v;N)$ of $v$  (or $cl(v)$ if $N$ is clear from the context) is the set of $x \in X$ which are descendants of $v$ along any directed path in $N$.  Here we utilize the ``hardwired'' clusters in the sense of \cite{hus10}.  The clusters are important for interpreting biological networks.  For example, in a biological network $cl(v)$ is the set of extant species whose genome possibly contains mutations originating in the ancestral species $v$.
An example of a network $N$ is shown in Figure 1 with $X = \{1,2,3,4\}$.  In $N$, $cl(8) = \{1,2,3\}$ and $cl(3) = \{3\}$.    

\begin{figure}[!htb]  
\begin{center}

\begin{picture}(200,145) (0,0)
\put(40,40){\circle{4}}
\put(40,40){\vector(-1,-1){28}}
\put(40,40){\vector(1,-1){28}}
\put(70,70){\circle{4}}
\put(70,70){\vector(-1,-1){28}}
\put(70,70){\vector(1,-1){28}}
\put(100,40){\circle{4}}
\put(100,40){\vector(-1,-1){28}}
\put(100,40){\vector(1,-1){28}}
\put(130,130){\circle{4}}
\put(130,130){\vector(-1,-1){58}}
\put(130,130){\vector(0,-1){118}}
\put(130,130){\vector(1,-3){29}}
\put(160,40){\circle{4}}
\put(160,40){\vector(1,-1){28}}
\put(160,40){\vector(-1,-1){28}}
\put(10,10){\circle{4}}
\put(70,10){\circle{4}}
\put(130,10){\circle{4}}
\put(190,10){\circle{4}}

\put(0,10) {1}
\put(55,10) {2}
\put(115,10){3}
\put(175,10){4}
\put(25,40){5}
\put(85,40){6}
\put(170,40){7}
\put(60,70){8}
\put(140,130){9}

\end{picture}

\caption{ A distinct-cluster phylogenetic network $N$. 
 }  
\end{center}
\end{figure}
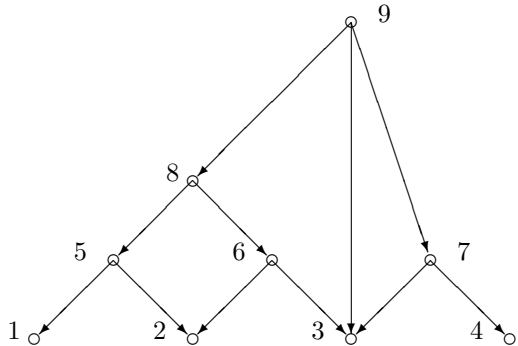

A network $N$ with vertex set $V$ is \emph{distinct-cluster} (DC) provided that whenever $v$ and $w$ are distinct vertices, then $cl(v) \neq cl(w)$.   Thus distinct vertices have different clusters.  Let $DC(X)$ denote the set of distinct-cluster acyclic rooted networks with leaf set $X$ up to isomorphism.  The network in Figure 1 lies in $DC(X)$.   

This paper studies $DC(X)$, which includes all trees, regular networks, and normal networks with leaf set $X$. 

It is useful to have a quantitative measure of the difference between two networks with the same leaf set $X$.  In comparing two trees, there are several metrics, including the Robinson-Foulds metric \cite{rob81}, the nearest neighbor interchange metric \cite{rob71}, \cite{wat78},  the triples distance \cite{cri96}, the SPR distance \cite{all01}, and the matching distance \cite{lin12}.  

Generalizing to networks that are not necessarily trees,  \cite{bar04} gives a metric between two regular networks.  The path-multiplicity distance or $\mu$-distance is  a metric between two tree-child networks  \cite{crv09}  and also between two networks that are tree-sibling and also time-consistent \cite{clr08}.  The tripartition distance \cite{car09} is a metric on rooted phylogenetic networks that are tree-child and also time-consistent.  The nested-labels distance \cite{clr09} is a metric on tree-child networks.

This paper defines a metric on $DC(X)$.  If $N_1$ and $N_2$ are in $DC(X)$, the \emph{inheritance metric} $D(N_1,N_2)$ between them is defined in Section 4.  The main tool is a matrix $H$ (or $H(N)$ to specify the network)  for $N\in DC(X)$  such that for each pair of vertices $u$ and $v$, $H_{u,v}(N)$ is the number of distinct directed paths in $N$ from $u$ to $v$.  Since $N$ is DC, we may identify $u$ and $v$ with their clusters.   Now suppose $N_1$ and $N_2$ are in $DC(X)$.  Suppose that  vertices $u_1$ and $v_1$ of $N_1$ and vertices $u_2$ and $v_2$ of $N_2$ satisfy that $cl(u_1;N_1) = cl(u_2;N_2)$ and $cl(v_1;N_1) = cl(v_2;N_2)$.  Then $u_1$ and $u_2$ may be identified because they have the same cluster and both networks are DC.   Similarly $v_1$ and $v_2$ may be identified.  A direct comparison is now possible between $H_{u_1,v_1}(N_1)$ and $H_{u_2,v_2}(N_2)$ since both count the number of directed paths between $u_1=u_2$ and $v_1=v_2$ in their respective networks.  The term $|H_{u_1,v_1}(N_1)-H_{u_2,v_2}(N_2)|$ contributes to $D(N_1,N_2)$.  

It is interesting that the matrix $H$ includes the information utilized in \cite{crv09} to produce a metric on tree-child networks.  More precisely, if $N$ is both DC and tree-child, then the vectors $\mu(v) $ utilized in \cite{crv09} have the entries $H_{v,x}$ for $x \in X$.  

Even DC networks can be very large.  The number of vertices in a member of $DC(X) $ can 
grow exponentially with $|X|$, bounded by the number of distinct nonempty subsets of $X$.  
Huge networks can be difficult to interpret, and it may be useful to ``simplify" a network into one with fewer vertices and arcs that is potentially easier to understand and which summarizes certain aspects of the original network.  For example, one might ask for a natural procedure to simplify a huge network into a tree or into a tree-child network.  Alternatively, one might ask which trees or normal networks $N_2$ ``best fit'' a given network $N_1$ in the sense of minimizing $D(N_1,N_2)$.  

This paper describes one simplification process:   If $N$ is in $DC(X)$, a \emph{cluster-preserving simplification} (CPS) of $N$ is a member of $DC(X) $ obtained recursively by removing a vertex or an arc from $N$ while taking care that each remaining vertex has its cluster unchanged (preserved).  In Section 5 we give definitions of the basic operations (defined so as not to change the clusters).  

In Section 6 we prove that to a large extent the order of the operations is immaterial. An arc $(u,v)$ is \emph{redundant} if there is a directed path from $u$ to $v$ other than the path along the arc $(u,v)$ itself.  For example, in Figure 1, $ (9,3) $ is redundant.    A key result  is Theorem 7.3, which asserts that when $V'$ is a subset of $V(N) $ containing the root $r$ and each member of $X$, then there exists a unique $N' \in DC(X) $ such that 
$V(N')  = V'$,
$N'$ has no redundant arcs, and 
 $N'$ is a CPS of $N$.

In particular, Theorem 7.3 says that the order of the operations in forming the CPS does not matter when the CPS has no redundant arcs.  Thus each member of a large family of CPS of $N$ depends only on its vertex set.  Moreover, Theorem 7.4 asserts that for two such subsets $V'$ and $V''$, if the unique networks are $N'$ and $N''$ respectively, then $N''$ is a CPS of $N'$ iff $V'' \subseteq V'$.  
An example is presented in Section 8, while some extensions  are described in Section 9. 

\section { Basic notions } 

A \emph{directed graph} $(V,E)$ is a set $V$ of \emph{vertices} and a set $E \subset V \times V$  where each $(a,b) \in E$ is called an \emph{arc}.  We interpret $ (a,b) \in E$ as a line segment directed from $a$ to $b$.  We assume there is no arc $(a,a)$, so that there are no loops.   Since $E$ is a set, not a multiset, there are no multiple arcs.  

If $a$ and $b$ are vertices, a \emph{directed path} from $a$ to $b$ is a sequence $a=v_0, \cdots , v_k=b$ in $V$ such that for $i = 0, \cdots , k-1, $ there is an arc $(v_i, v_{i+1}) \in E$.  The \emph{length} of the path is $k$.  There is always a directed path of length 0 from $a$ to $a$.  An arc $ (a,b) $ is \emph{redundant} if there exists a directed path $a = v_0, v_1, \cdots , v_k = b$ with $k > 1$ (so that the path is not the same as the path consisting of the single arc $ (a,b) $).  

If $q$ and $c$ are vertices then $c$ is a \emph{child} of $q$ and $q$ is a \emph{parent} of c iff there exists an arc $(q,c)\in E$.  A child $c$ of $q$ is a \emph{tree-child} iff $q$ is the only parent of $c$.  A vertex is \emph{hybrid} if it has more than one parent.  The \emph{out-degree} of a vertex $v$ is the number of children of $v$, while the \emph{in-degree} of a vertex $v$ is the number of parents of $v$.   A vertex is a \emph{leaf} if it has no child. 

Let $X$ be a set (for example, of biological species).  If $L$ is the set of leaves, we shall assume there is a bijection $\psi: X \to L$.  Usually we will identify $x \in X$ with $\psi(x) \in L$, so if $x\in X$ we may write $x\in L$.  Let $\mathcal{P}(X)$ denote the set of all nonempty subsets of $X$.   

A directed graph $(V,E)$ is \emph{rooted} with \emph{root} $r$ if there exists a vertex $r$ such that for every $v \in V$ there is a directed path from $r$ to $v$.  In Figure 1, $N$ is rooted with root $r = 9$.

A directed graph is \emph{acyclic} if there is no directed cycle; i.e., there is no vertex $v$ with a directed path from $v$ to $v$ of length greater than 0.  If the network is both rooted and acyclic, then the root is necessarily unique.  

If $N=(V,E)$ is acyclic, write $u \leq v$ iff there is a directed path in $N$ from $u$ to $v$.  Since $N$ is acyclic it follows that if $u \leq v$ and $v \leq u$ then $u = v$.  Transitivity is obvious, so $\leq$ is a partial order on $V$.  If $u \leq v$ and $u \neq v$, write $u < v$.  

An \emph{$X$-cluster} is a nonempty subset of $X$.  If $v$ is a vertex, then the \emph{cluster} $cl(v)$ of $v$ (or $cl(v; N) $ if we need to specify the network $N$) is $cl(v) = \{x \in X: v \leq x\}$.  It is thus an $X$-cluster containing the leaves $x \in X$ such that there exists a directed path from $v$ to $x$ (of any length).  Note that $cl(v)$ is the ``hardwired'' cluster of $v$ in the sense of \cite{hus10}.  Since there are no directed cycles, a directed path of maximum length starting at $v$ must end at a leaf, so that $cl(v)$ is nonempty.   If $x \in X$, then $cl(x) = \{x\}$ via the directed path from $x$ to $x$ of length 0.  In Figure 1, $cl(6) = \{2,3\}$ and $cl(9) = \{1,2,3,4\}$.  

A cluster is \emph{trivial} if it has the form $X$ or $\{x\}$ where $x \in X$.  Let $Tr(X)$ denote the set of trivial clusters for the set $X$.  In Figure 1, $Tr(X) = \{\{1\}, \{2\}, \{3\},$ $ \{4\}, \{1,2,3,4\}\}$.   

It is easy to see that if $u \leq v$ then $cl(v) \subseteq cl(u) $. 

A network is \emph{distinct-cluster} (DC) if no two vertices have the same cluster; i.e., whenever $u\neq v$ then $cl(u) \neq cl(v)$.  
If the network is DC, then each vertex that is not a leaf has out-degree at least 2.  To see this, note that  if $v$ has child $c_1$ then $cl(c_1 ) \subset cl(v) $, so there exists $x \in cl(v) - cl(c_1)$ and there must exist a child $c_2$ of $v$ such that $x \in cl(c_2)$.  The network in Figure 1 is distinct-cluster.

We summarize the properties we shall require in the definition of an $X$-network:  
An \emph{X-network} $N$ is $N = (V,E,r,X) $
where $ (V,E) $ is a finite acyclic directed network  with root $r$ and leaf set $X$.  We may write $V(N)$ for $V$ or $E(N)$ for $E$.  Two $X$-networks $N_1$ and $N_2$ are \emph{isomorphic} if there is a bijection $\phi: V(N_1) \to V(N_2)$ such that $(\phi(u),\phi(v)) \in E(N_2)$ iff $(u,v)\in E(N_1)$ and such that the labels of the leaves are preserved.  A \emph{DC X-network} is an $X$-network that is DC.  

There are several special kinds of networks that are of interest.  Let $N = (V,E,r,X) $ be an $X$-network.   Then $N$ is a \emph{tree} iff no vertex is hybrid.  $N$ is \emph{tree-child} \cite{crv09} iff every vertex $v$ that is not a leaf has a child $c$ that is tree-child. $N$ is \emph{normal}  \cite{wil10} iff $N$ is tree-child, no vertex has outdegree 1, and in addition $N$ has no redundant arc. $N$ is \emph{regular} \cite{bar04}  iff  (1) it is DC and (2) there exists an arc $ (u,v) $ iff both 
$cl(v) \subset cl(u) $ and
there is no vertex $w$ such that $cl(v) \subset cl(w) \subset cl(u)$.

Every tree, every regular network, and every normal network is DC.   There exist tree-child networks which are not in $DC(X)$ (for example, Figure 2 of \cite{crv09}).   
Figure 1 is DC.  It is not tree-child and not normal since 6 has only hybrid children.  It is not regular since it contains the redundant arc $(9,3)$.

\section { The inheritance matrix} 

Let $N = (V,E,r, X) $ be an $X$-network (not necessarily DC).  Let the vertices be $v_1, \cdots , v_m$ in some order.
The \emph{adjacency matrix} $A$ of $N$ is the $m \times m$ matrix \\
	\begin{gather*}
	A_{i,j} = 
\begin{cases}
	1  & \text{  if there is an arc  } (v_i,v_j)\text{  in }E\\
	0 & \text{ otherwise}
\end{cases}
\end{gather*}

Note that $A$ encodes all the structure of $N$.     Given $A$, the leaves are those $v_i$ such that for all $j$, $A_{i,j} = 0$.  There is an arc $ (v_i, v_j) $ iff $A_{i,j} > 0$.  

It is well known (see, for example, \cite{har69})  that for all $k>0$,  $ (A^k)_{i,j}$ is the number of directed paths of length $k$ from $v_i$ to $v_j$.   

In this section we define a matrix $H$, called the \emph{inheritance matrix},  which will be of use in defining a metric on the set of DC $X$-networks.  An example is given in Section 8.

Since $N$ is finite and acyclic, there is a directed path of maximum length. If the maximum length is $L$, there are no paths of length $L+1$.

\begin{lem} 
Let the maximum length of a directed path in $N$ be $L$.  Then $A^k = 0$ for $k \geq L+1$.
 \end{lem}

\begin{proof}
If $k \geq L+1$, then $N$ has no directed path of length $k$. 
Hence $A^k = 0$. 
\end{proof}

Let $L$ be the maximum length of a directed path.  Define\\
$$H = I + A + A^2 + \cdots  + A^L. $$
By Lemma 3.1, it is equivalent to write
$$H = I + \sum_{j=1}^{\infty} A^j. $$

Call $H$ the \emph{inheritance matrix} (from the ``h'' in ``inheritance'').  If we need to specify the network $N$, we may also write $H(N)$ instead of $H$.

\begin{thm} 
For all vertices $v$, $H_{v,v} = 1$.
\end{thm}

\begin{proof}
Let $v$ be a vertex.  Since $N$ is acyclic, no $k > 0$ satisfies $A^k_{v,v} > 0$.  It follows that $H_{v,v} = I_{v,v} = 1$.
\end{proof}

\begin{thm} 
For all vertices $u$ and $v$, $H_{u,v}$ is the number of directed paths in $N$ from $u$ to $v$ of any nonnegative length.
\end{thm}

\begin{proof}
By the definition, $H_{u,v} = I _{u,v}+ \sum_{j=1}^{\infty} (A^j)_{u,v}. $ 
For each $j>0$, $(A^j)_{u,v}$ is the number of directed paths of length $j$ from $u$ to $v$. The trivial path $u$ of length 0 from $u$ to $u$ corresponds to the term  $I_{u,u}$. 
\end{proof}

\begin{thm} 
$H$ is invertible, and $I - A = H^{-1}$.  
\end{thm}

\begin{proof}
 Let $L$ be the maximum length of a directed path.   Then $A^{L+1} = 0$ and 
$ (I-A) H$
$= (I-A)(I+A+A^2 + A^3 + \cdots  + A^L) $
$= I - A^{L+1}$ [by telescoping] $= I$.
\end{proof}

\begin{cor} 
$H = (I-A)^{-1}$  and  $A =  I - H^{-1}$.
\end{cor}

\begin{cor} 
$N$ can be reconstructed given either $A$ or $H$.
\end{cor}

\begin{proof}
$N$ can be reconstructed from $A$ since $A_{u,v} > 0$ iff there is an arc $ (u,v) $.  But by Corollary 3.5, $A$ can be found from $H$.  Hence $N$ can be reconstructed from $H$ as well. 
\end{proof}

\begin{thm} 
Suppose $N$ has $A$ and $H$ as above.  Then\\
(i) $H-I = AH = HA$.\\
(ii) If $u \neq v$ and $u$ has children $c_1, \cdots , c_k$ then
$H_{u,v} = \sum_{i=1}^k H_{c_i, v}$.\\
(iii) If $u \neq v$ and $v$ has parents $q_1, \cdots , q_k$ then
$H_{u,v} = \sum_{i=1}^k H_{u, q_i}$.
\end{thm}

\begin{proof}
 (i)
$H-I = A + A^2 + \cdots  + A^L$   where $A^{L+1} = 0$\\
$= A + A^2 + \cdots  + A^L + A^{L+1}$ 
$= A( I + A + \cdots  + A^L) = AH$.
A similar argument applies for $HA$.

(ii) From (i), $H-I =AH$.
Since $u \neq v$, 
$H_{u,v} = (H-I)_{u,v} = (AH)_{u,v}  = \sum A_{u,w} H_{w,v}$ where $w$ ranges over the vertices of $N$.
But  if $A_{u,w} > 0$, $w$ can only have values $c_i$, and $A_{u,c_i} =1$.

(iii) From (i), $H-I = H A$.
Since $u \neq v$, 
$H_{u,v} = (H-I)_{u,v} = (HA)_{u,v} =  \sum H_{u,w} A_{w,v}$ where $w$ ranges over the vertices of $N$.  
But if $A_{w,v} >0$, then $w$ can only have values $q_i$ and $A_{q_i, v} = 1$.
\end{proof}

Note that (ii) says that each path from $u$ to $v$ must go through an initial arc $(u,c_i)$.  Similarly  (iii) says that each path from $u$ to $v$ must go through a final arc $(q_i, v) $.

\section { The inheritance metric on DC $X$-networks} 

Fix $X$ and let $DC(X) $ denote the set of isomorphism classes of DC $X$-networks.  We regard each member of $DC(X) $ as a rooted acyclic directed $X$-network $N=(V,E,r,X) $ that is distinct-cluster.   Two members of $DC(X)$ are considered equal if and only if they are isomorphic.  

Let $N = (V,E,r,X) $ and $N' = (V',E', r', X) $ be in $DC(X) $.  Let $A$ and $A'$ be the adjacency matrices of $N$ and $N'$ respectively, and let $H$ and $H'$ be the inheritance matrices of $N$ and $N'$ respectively.  Since $N$ is DC, we may identify each vertex $u$ with its cluster $cl(u;N)$, so that $H_{u,v}$ is defined if $u$ and $v$ are clusters from $X$ that occur as clusters of vertices of $N$.  In particular if $u$ and $v$ are clusters in both $N$ and $N'$, then both $H_{u,v}$ and $H'_{u,v}$ make sense.  Moreover both $r$ and $r'$ are identified with the cluster $cl(r;N) = X$ while each leaf $x \in X$ is identified with the singleton set $\{x\}$.  

Let $S$ be the set of all clusters of $N$ and $S'$ the set of all clusters of $N'$.  Let $C$ denote a subset of $\mathcal{P}(X)$ that contains both $S$ and $S'$; thus $C$ is a collection of nonempty subsets of $X$ containing $S\cup S'$.   Define the \emph{inheritance matrix $_CH$ of $N$ over $C$} as follows if $u$ and $v$ are in $C$:

	\begin{gather*}
	_CH_{u,v} = 
\begin{cases}
	H_{u,v}  & \text{ if both } u \text{ and } v \text{ are in } S\\
	0 & \text{ otherwise}
\end{cases}
\end{gather*}

Note that if $u \in C$ but $u \notin S$ then $_CH_{u,u} = 0$.  

Similarly we define the inheritance matrix $_CH'$ of $N'$ over $C$.  
Define 
$$_CD(N,N') = \sum \Big[|_CH_{u,v} - \,_CH'_{u,v} | : u \in C, v \in C\Big].$$

The following lemma says that $\,_CD(N,N') $ does not depend on the choice of $C$.   We thus obtain the same number using different choices of $C$.  

\begin{lem} 
Let $N$ and $N'$ be DC $X$-networks with cluster sets $S$ and $S'$ respectively.  Let $C$ and $C'$ be sets of $X$-clusters that contain both $S$ and $S'$.  Then 
$_CD(N,N') = \,_{C'}D(N,N') $.
\end{lem}

\begin{proof}
 The only nonzero terms $|_CH_{u,v} - \,_CH'_{u,v} |$ or $|_{C'}H_{u,v} - \,_{C'}H'_{u,v} |$ in the definitions occur when 
(a) both $u$ and $v$ are in $S$ or 
(b) both $u$ and $v$ are in $S'$.  
In either case, $|_CH_{u,v} - \,_CH'_{u,v} |=|_{C'}H_{u,v} - \,_{C'}H'_{u,v} |$.   
\end{proof}

Define $D(N,N') $ to be the common value $_CD(N,N') $ for any choice of $C$ containing both $S$ and $S'$.  We shall see below that $D(N,N') $ is a metric, so we will call $D(N,N') $ the \emph{inheritance metric} or \emph{inheritance distance} between $N$ and $N'$.  Most commonly one will utilize $C=S\cup S'$, but different choices may be convenient when one is comparing several different networks.  

Roughly, $D(N,N')$ counts the sum, over all $u$ and $v$, of the absolute value of the difference between the number of directed paths from $u$ to $v$ in $N$ and the number of directed paths from $u$ to $v$ in $N'$.   

\begin{thm} 
Fix $X$.  Then  $D$ defines a metric on $DC(X) $.  In particular for any DC $X$-networks $N$, $N'$, and $N''$ we have\\
(i) $D(N,N') \geq 0$.\\
(ii) $D(N,N') = D(N', N) $.\\
(iii) $D(N,N' ) \leq D(N, N'') + D(N'', N') $.\\
(iv) $D(N,N') = 0$ iff $N = N'$.
\end{thm}

\begin{proof}
Pick any $C$ that contains both $S$ and $S'$.  It is immediate that $_CD(N,N') \geq 0$ and that $_CD(N,N') = _CD(N',N) $.  If $N''$ is a third $X$-network with cluster set $S''$ and $C$ contains $S$, $S'$, and $S''$, then when $u$ and $v$ range over members of $C$ we have\\
$_CD(N,N') = \sum |_CH_{u,v} - \,_CH'_{u,v} |$\\
$= \sum |\,_CH_{u,v} -\,_CH''_{u,v} + \,_CH''_{u,v} -  \,_CH'_{u,v} |$\\
$\leq \sum \big[ |\,_CH_{u,v} -\,_CH''_{u,v}| + |\,_CH''_{u,v} -  \,_CH'_{u,v} |\big]$\\
$= \sum |\,_CH_{u,v} -\,_CH''_{u,v}| +\sum |\,_CH''_{u,v} - \, _CH'_{u,v} |$\\
$= \,_CD(N,N'') + \,_CD(N'',N'),$\\
proving the triangle inequality (iii).  

There remains only to show that  $_CD(N,N') = 0$  iff $N = N'$.  It is immediate that   $_CD(N,N) = 0$.   Conversely, suppose $_CD(N,N') = 0$.  If $|X|= 1$, then the  network with only a single vertex is the only DC $X$-network, whence $N = N'$.  Thus we may assume $|X| > 1$.  

Since $\,_CD(N,N') = 0$ it follows that for each $u$ and $v$ in $C$, $|\,_CH_{u,v} - \,_CH'_{u,v} | = 0$ so  
$\,_CH_{u,v} =  \,_CH'_{u,v}.$  
In particular for each $u$,  $\,_CH_{u,u} = \,_CH'_{u,u}$.  But $u \in S$ iff  $\,_CH_{u,u} = 1$ and $u \in S'$ iff  $\,_CH'_{u,u} = 1$ by Theorem 3.3.  Hence $S = S'$, so $N$ and $N'$ have the same vertex sets.  

If both $u$ and $v$ are in $S = S'$, then
since $\,_CH_{u,v} = \,_CH'_{u,v}$
it follows that  $H_{u,v} = H'_{u,v}$.   Thus $H = H'$.  By Corollary 3.5,
$A =  I - H^{-1} = I - H'^{-1} = A'$.
It follows that $N = N'$.
\end{proof}

The following Theorem 4.3 tells the distance between two DC $X$-networks chosen to be especially distant from each other.  

The \emph{trivial tree} $Tr(X)$ is the tree whose clusters are all the trivial clusters $Tr(X)$ with one arc $ (X,x) $ for each $x \in X$.  
The \emph{regular network of all nonempty subsets of $X$} is $\mathcal{P}(X) $ with vertex set $\mathcal{P}(X)$ and an arc $ (A,B) $ iff $B \subset A$ and there is no $C \in \mathcal{P}(X) $ such that $B \subset C \subset A$.  

\begin{thm} 
Let $N$ and $N'$ be DC $X$-networks, where $|X| = n$.  Then\\
(1) $D(N,N') $ is an integer.  \\
(2) Suppose $Tr(X)$ is the trivial tree on $X$ and $\mathcal{P}(X) $ is the regular $X$-network of all nonempty subsets of X.   Then for $n \geq 2$, 
$$D(\mathcal{P}(X),Tr(X)) = \sum_{k=1}^n {n \choose k} \sum_{j=1}^k {k \choose j} (k-j)! - 2n - 1.$$
\end{thm}

\begin{proof}
 (1) Since both $_CH_{u,v}$ and $_CH'_{u,v}$ are integers, (1) is immediate.

(2) We first observe that in $\mathcal{P}(X) $, if $A \subset B$ and $|B| - |A| = k$, then the number of paths from $B$ to $A$ in $N$  is $k! $.  To see this, list the elements of $B-A$ as 1, $\cdots$, $k$.  Each arc has as its endpoint a subset containing one fewer member of $X$ than the subset at the starting point of the arc. Denote the path by $B = S_0, S_1, S_2, \cdots , S_k = A$.  But there are $k$ ways to choose $S_1$ (deleting one member), then $ (k-1) $ ways to choose $S_2$, then $ (k-2) $ ways to choose $S_3$, etc.  The total number of paths is then $k! $.

Now we find $\sum H_{B,A}$ for all possible $B$ and $A$.
If the set $B$ has size $k$ then the number of ways to choose $B$ is ${n \choose k}$, where $1 \leq k \leq╩n$.  If $B$ has been chosen, then the number of ways to choose $A$ is ${k \choose j}$ , where $j = |A|$, and $1 \leq j \leq╩k$.  The number of paths from $B$ to $A$ is then $ (k-j)! $.  Hence the total number of paths is
$$\sum_{k=1}^n {n \choose k} \sum_{j=1}^k {k \choose j} (k-j)!.$$

Note that $Tr(X)$ has one arc $ (X,x) $ for each $x\in X$.  Hence $H_{u.v}(Tr(X)) = 0$ unless $u = r$ and $v \in X$ or else $u = v = r$, or else $u = x = v$ for some $x \in X$.  In any of these cases $H_{u.v}(Tr(X)) = 1$ .  Hence the total sum of the entries is 
$n + 1 + n = 2n+1$.

The difference yields (2) since in any case when $H_{u.v}(Tr(X)) = 1$ then\\ $H_{u.v}(\mathcal{P}(X)) \geq 1$.  
\end{proof}

If $B_n = D(\mathcal{P}(X),Tr(X)) $ with $|X|=n$, then $B_2 = 0$, $B_3 = 15$, $B_4 = 94$, $B_5 = 535$, $B_6 = 3287$.  I conjecture that for any $n\geq 2$, $D(\mathcal{P}(X),Tr(X)) $ is the maximum value of any $D(N,N') $.

\section { Cluster-preserving simplifications} 

Let $N  = (V,E,r,X) $ be a DC $X$-network.   We consider two simplifying steps:

(1) Suppose $v \in V$, $v \notin X$, $v \neq r$. We delete vertex $v$ by \emph{passing through} to form a new network $N'$ as follows:
Let $q_1, \cdots , q_k$ denote all the parents of $v$, and let $c_1, \cdots ,c_m$ denote all the children of $v$.  Remove $v$ and all arcs involving $v$ from $E$.
Add new arcs $ (q_i,c_j) $ for $i = 1, \cdots , k$; $j = 1, \cdots , m$.  (If one already exists, then we retain just the single copy.)  Let $N'$ denote the result.  
Thus $N' = (V', E', r, X) $ where $V' = V - \{v\}$, $E' = \big[E - \{(q_i,v), (v,c_j)\}\big] \cup \{(q_i,c_j)\} $.  
Alternatively we will denote the result as $N' = D(v) N$, meaning deletion of $v$ from $N$.

(2) Let $ (a,b) \in E$ be a redundant arc in $N$.  
Form $N' = (V,E',r,X) $ where 
$E' = E - \{(a,b)\} $.
Thus $N'$ is $N$ with arc $ (a,b) $ removed.
Alternatively we will denote the result as $D(a,b) N$, meaning deletion of the redundant arc $(a,b)$ from $N$.

Figure 2 shows a DC $X$-network $N$ with $X = \{1,2,3,4\}$ and redundant arc $(10,4) $. The networks $D(6)N$ and $D(10,4)N$ are also shown.  Vertices continue to be labeled with the same labels of $N$ when they have the same clusters.  Note that $D(6)N$ contains the redundant arcs $(8,2)$ and $(9,3)$ which were not present in $N$.  In general, deleting a vertex can introduce new redundant arcs.

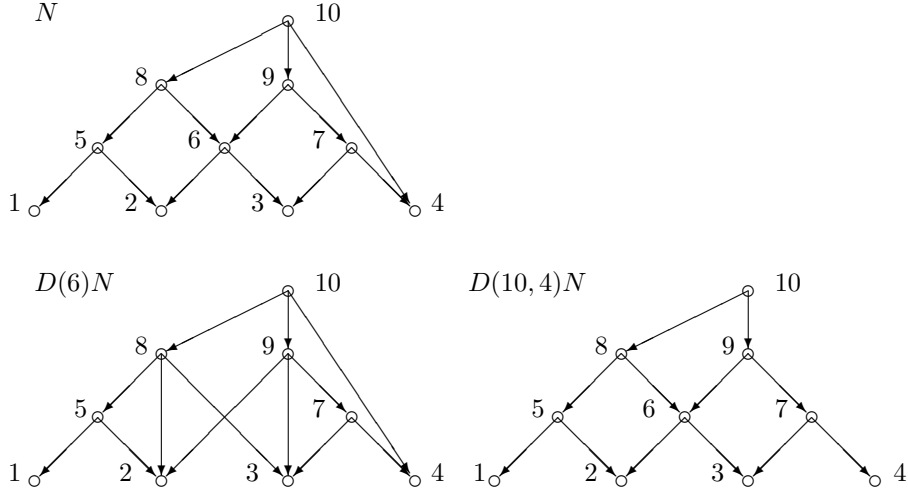
\begin{figure}[!htb]  
\begin{center}

\begin{picture}(320,200) (0,0)  

\put(34,136){\vector(-1,-1){22}}
\put(34,136){\circle{4}}
\put(34,136){\vector(1,-1){22}}
\put(58,160){\vector(-1,-1){22}}
\put(58,160){\vector(1,-1){22}}
\put(58,160){\circle{4}}
\put(106,160){\vector(1,-1){22}}
\put(106,160){\vector(-1,-1){22}}
\put(106,160){\circle{4}}
\put(82,136){\vector(-1,-1){22}}
\put(82,136){\vector(1,-1){22}}
\put(82,136){\circle{4}}
\put(130,136){\vector(-1,-1){22}}
\put(130,136){\vector(1,-1){22}}
\put(130,136){\circle{4}}
\put(106,184){\vector(-2,-1){46}}
\put(106,184){\vector(0,-1){22}}
\put(106,184){\vector(2,-3){46}}
\put(106,184){\circle{4}}
\put(106,112){\circle{4}}
\put(58,112){\circle{4}}
\put(10,112){\circle{4}}
\put(154,112){\circle{4}}

\put(0,112) {1}
\put(44,112) {2}
\put(92,112){3}
\put(160,112){4}
\put(25,136){5}
\put(68,136){6}
\put(115,136){7}
\put(48,160){8}
\put(96,160){9}
\put(116,184){10}
\put(10,184){$N$}

\put(34,34){\vector(-1,-1){22}}
\put(34,34){\vector(1,-1){22}}
\put(34,34){\circle{4}}
\put(58,58){\vector(-1,-1){22}}
\put(58,58){\vector(1,-1){46}}
\put(58,58){\vector(0,-1){46}}
\put(58,58){\circle{4}}
\put(106,58){\vector(1,-1){22}}
\put(106,58){\vector(-1,-1){46}}
\put(106,58){\vector(0,-1){46}}
\put(106,58){\circle{4}}
\put(130,34){\vector(-1,-1){22}}
\put(130,34){\vector(1,-1){22}}
\put(130,34){\circle{4}}
\put(106,82){\vector(-2,-1){46}}
\put(106,82){\vector(0,-1){22}}
\put(106,82){\vector(2,-3){46}}
\put(106,82){\circle{4}}
\put(106,10){\circle{4}}
\put(58,10){\circle{4}}
\put(10,10){\circle{4}}
\put(154,10){\circle{4}}

\put(0,10) {1}
\put(42,10) {2}
\put(90,10){3}
\put(160,10){4}
\put(25,34){5}
\put(115,34){7}
\put(48,58){8}
\put(96,58){9}
\put(116,82){10}
\put(10,82){$D(6)N$}

\put(208,34){\vector(-1,-1){22}}
\put(208,34){\vector(1,-1){22}}
\put(208,34){\circle{4}}
\put(232,58){\vector(-1,-1){22}}
\put(232,58){\vector(1,-1){22}}
\put(232,58){\circle{4}}
\put(280,58){\vector(1,-1){22}}
\put(280,58){\vector(-1,-1){22}}
\put(280,58){\circle{4}}
\put(256,34){\vector(-1,-1){22}}
\put(256,34){\vector(1,-1){22}}
\put(256,34){\circle{4}}
\put(304,34){\vector(-1,-1){22}}
\put(304,34){\vector(1,-1){22}}
\put(304,34){\circle{4}}
\put(280,82){\vector(-2,-1){46}}
\put(280,82){\vector(0,-1){22}}
\put(280,82){\circle{4}}
\put(184,10){\circle{4}}
\put(232,10){\circle{4}}
\put(280,10){\circle{4}}
\put(328,10){\circle{4}}

\put(176,10) {1}
\put(218,10) {2}
\put(266,10){3}
\put(335,10){4}
\put(198,34){5}
\put(240,34){6}
\put(290,34){7}
\put(222,58){8}
\put(270,58){9}
\put(290,82){10}
\put(174,82){$D(10,4)N$}

\end{picture}

\caption{  $N$  with $D(6)N$ and $D(10,4)N.$ 
 }  
\end{center}
\end{figure}

\begin{thm} 
Let $N = (V,E,r,X)$ be a DC $X$-network.  Let $v$ be a vertex, $v \neq r$, $v \notin X$.  Let $ (a,b) $ be a redundant arc of $N$.\\  
(1) For vertices $u$ and $w$ distinct from $v$, there is a directed path from $u$ to $w$ in $D(v)N$ iff there is a directed path in $N$ from $u$ to $w$.\\
(2) $D(v)N$ is a DC $X$-network.\\
(3) For every vertex $w$ in $N$ if $w \neq v$, 
$cl(w; D(v)N) = cl(w; N) $.\\
(4) For vertices $u$ and $w$, there is a directed path from $u$ to $w$ in $D(a,b)N$ iff there is a directed path in $N$ from $u$ to $w$.\\
(5) $D(a,b)N$ is a DC $X$-network.\\
(6)  For every vertex $w$ in $N$,
$cl(w; D(a,b)N) = cl(w;N) $.\\
(7) Suppose $u$ and $w$ are vertices of $N$ other than $v$.  Then $u < w$ in $N$ iff $u < w$ in $D(v) N$.\\
(8) Suppose $u$ and $w$ are vertices of $N$.   Then $u < w$ in $N$ iff $u < w$ in $D(a,b) N$.  
\end{thm}

\begin{proof}
If $N' = D(v)N$, then $ (V',E') $ is a directed graph. By the restriction on the choice of $v$, it follows that $V'$ contains $r$ and each member of $X$.  

We first show (1).  Suppose $u$ and $w$ are vertices of $N$ distinct from $v$.   First suppose that there is a directed path from $u$ to $w$ in $N$.  If $u = v_0, v_1, \cdots , v_k = w$ is a directed path in $N$ from $u$ to $w$, then the same path is a directed path from $u$ to $w$ in $N'$ if no vertex $v_i$ equals $v$.  If, on the other hand  $v_j = v$ for some $0 < j < k$  then $v_{j-1}$ is a parent of $v$ and $v_{j+1}$ is a child of $v$.  Hence the directed path  $u = v_0,  v_1, \cdots , v_{j-1}, v_{j+1}, \cdots , v_k = w$ is a directed  path from $u$ to $w$ in $N'$.  

Conversely let $u = v_0, v_1, \cdots , v_k = w$ be a directed path in $N'$.  The only way this could fail to be a path in $N$ is if, for some $j$, $ (v_{j-1}, v_j) $ is not an arc of $N$.  By construction this means that $v_{j-1}$ is a parent of $v$ and $v_j$ is a child of $v$.  Hence $u = v_0, v_1, \cdots , v_{j-1}, v, v_j, \cdots , v_k = w$ is a directed path from $u$ to $w$ in $N$.  This proves (1).  

Since $N$ was acyclic, it follows that $N'$ is acyclic because a cycle starting and ending at $w$ in $N'$ would imply a cycle in $N$ as well.   (3) follows from (1) since for $x \in X$, $x \in cl(u;N) $ iff there is a directed path in $N$ from $u$ to $x$, and similarly for $N'$.  But then $N'$ is distinct-cluster since $N$ was. In addition $r$ remains a root of $N'$ since for each vertex $w$ there remains a path from $r$ to $w$. Thus (2) is true.

Now let $N' = D(a,b) N$.  Thus $N' = (V,E') $ where  $E' = E - \{(a,b)\} $.  Note that  $(V,E') $ is a directed graph. 

To prove (4), let $u$ and $w$ be vertices of $N$. Note that any directed path in $N'$ is also a directed path in $N$.  Conversely, suppose $u = v_0, v_1, \cdots , v_n = w$ is a path in $N$ from $u$ to $w$.  If there is no $j$  such that $ (a,b) = (v_j, v_{j+1})$, then it remains a path from $u$ to $v$ in $N'$.  If, however, $a = v_j$, $b = v_{j+1}$, then since $ (a,b) $ is redundant we may choose a path $a = u_0, \cdots , u_m = b$ from $a$ to $b$ in $N$ with $m > 1$ that does not include the arc $ (a,b) $.   Then
$u = v_0, \cdots, v_j=a=u_0, u_1, \cdots , u_m = b =v_{j+1}, v_{j+2}, \cdots, v_n=w$ is a path in $N'$ from $u$ to $w$.  This proves (4).

Since $N$ was acyclic, by (4) $N'$ is acyclic.
Note that $r$ remains a root of $N'$ since for every vertex $v$, $r \leq v$ in N, whence $r \leq v$ in $N'$.  Moreover if $u \notin X$, then $cl(u;N) = \{x \in X: u < x \text{ in } N\}$ $= \{x \in X: u < x \text{ in } N'\} = cl(u;N') $.  This proves (5) and (6).

(7) and (8) are restatements of (1) and (4) respectively.  
\end{proof}

Let $N$ and $N'$ be DC $X$-networks.  We call $N'$ a \emph{cluster-preserving simplification} (CPS) of $N$ provided there exists a sequence  $N = N_0, \cdots , N_k =N'$ of DC $X$-networks such that for $i = 0, \cdots , k-1$,
either $N_{i+1} = D(v)N_i$ for some $v \in V(N_i) $
or $N_{i+1} = D(a,b) N_i$ for some redundant arc $ (a,b) \in E(N_i) $.

\begin{thm} 
Let $N$ be a DC $X$-network and $N'$ be a CPS of $N$.  Then\\
(1) $N'$ is a DC  $X$-network.\\
(2) For each $v \in V(N') $, there exists a unique vertex $\phi(v) \in V(N) $ such that
$cl(\phi(v); N) = cl(v; N') $.\\
(3) If $ (u,v) $ is an arc of $N'$, then there is a directed path from $\phi(u) $ to $\phi(v) $ in $N$.\\
(4) If $u$ and $v$ are in $V(N') $ and there is a directed path from $\phi(u) $ to $\phi(v) $ in $N$, then there is a directed path from $u$ to $v$ in $N'$.\\
(5) Suppose $u$ and $w$ are vertices of $N'$.  Then $\phi(u) < \phi(w)$ in $N$ iff $u < w$ in $ N'$.\\
\end{thm}

\begin{proof}
Let $N = N_0, \cdots , N_k = N'$ be a sequence of DC $X$-networks such that for $i = 0, \cdots , k-1$, either $N_{i+1} = D(v)N_i$ for some $v \in V(N_i) $
or $N_{i+1} = D(a,b) N_i$ for some redundant arc $ (a,b) \in E(N_i) $.

(1) is immediate from Theorem 5.1.

For (2) given $1\leq j\leq k$, define $\phi_j: V(Nj) \to V(N_{j-1})$
by $\phi_j(w) =$ the unique vertex in $V(N_{j-1})$ such that $cl(\phi_j(w);N_{j-1}) = cl(w;N_j) $.  Define $\phi:V(N') \to V(N) $ by $\phi(v) =   \phi_1 \circ \phi_2 \circ \cdots   \circ \phi_k$.   By Theorem 5.1,  $cl(\phi(v); N) = cl(v; N') $.  Since $N$ is DC, there exists at most one vertex with a given cluster, proving uniqueness.

(3) and (4) follow from (7) and (8) of Theorem 5.1.  (5) follows from (3) and (4).  
\end{proof}

Observe that by Theorem 5.2(2), each vertex in $N'$ has the same cluster as the corresponding vertex in $N$.  This justifies the name ``cluster-preserving.'' Moreover, directed paths from $u$ to $v$ in $N'$ correspond to directed paths from $\phi(u) $ to $\phi(v) $ in $N$.  Thus much essential information about $N$ is preserved in $N'$.  

We will generally identify a vertex $v$ in $N'$ with the vertex $\phi(v) $ in $N$, so we may say that each vertex of a CPS $N'$ of $N$ is also a vertex of $N$.  

The number of vertices of $N'$ could be considerably smaller than the number of vertices of $N$.  In some cases the number of arcs in $N'$ could be greater than the number of arcs in $N$.  

\begin{thm} 
Suppose that the DC $X$-network $N$ is a tree.  Then every CPS of $N$ is also a tree.
\end{thm}

\begin{proof}
Let the CPS $N'$ of $N$ be given by $N' = D_k  \cdots   D_1 N$,  where each $D_i$ has form either $D(v) $ or $D(a,b) $.  We prove the theorem by induction on $k$.   If $k = 0$, the result is immediate.  Assume the result is true for $k$, and we will prove the result for $k+1$.  Now let the CPS $N'$ of $N$ be given by $N' = D_{k+1}  \cdots  D_1 N$.  We prove that $N'$ is a tree.

Since $N$ is a tree, it has no redundant arc, so there must be a vertex $v$ of $N$ such that $D_1 = D(v) $. Since $v \neq r$ , $v$ has at least one parent, and since $N$ is a tree, $v$ has at most one parent.   Let $q$ denote the unique parent of $v$.  Since $v \notin X$, $v$ has at least one child.  But if $v$ had only one child $c$ then $cl(v;N) = cl(c;N) $, contradicting that $N$ is distinct-cluster.  Hence $v$ has children $c_1, c_2, \cdots , c_m$ for some $m\geq 2$.  In $D(v)N$ vertex $v$ and all arcs incident with $v$ are deleted, while there are new arcs $ (q,c_i) $ for $i = 1, \cdots , m$.  Each $c_i$ had a unique parent $v$ in $N$ and now has a unique parent $q$ in $N'$.  Each vertex of $N'$ other than the root or $c_i$ has the same unique parent as in $N$.  Hence $D(v)N$ has no hybrid vertex and $D(v)N = D_1N$ is a tree.  Write $M = D_1 N$.   Now $N' = D_{k+1}  \cdots   D_2 M$.   Since there are $k$ factors and $M$ is a tree, it follows that $N'$ is a tree by the inductive hypothesis.  
\end{proof}

\section { Modifying the description of a CPS} 

A CPS $N'$ of $N$ can be written in the form  $N' = D_k \cdots   D_1 N$ where, letting  
$N_i = D_i \cdots  D_1 N$, each $D_i$ has form either $D(v) $ for $v$ a vertex of $N_{i-1}$ or $D(a,b) $  for $(a,b)$ a redundant arc of $N_{i-1}$.  Note that the redundant arc $(a,b)$ might not have been present in $N$ but rather have been introduced via some $D_j$ with $j<i-1$.  In this section we show that this description may be changed in various ways.  For example $D(v) D(w) N = D(w) D(v) N$ if both $w$ and $v$ are vertices of $N$ and $v \neq w$. 

\begin{thm} 
Let $N$ be a DC $X$-network. Suppose $v$ is a vertex and $ (a,b) $ is a redundant arc of $N$.  \\
(i)  If $w$ is a vertex and $w \neq v$, then $D(v) D(w) N = D(w) D(v) N$.\\
(ii) Assume $v \neq a$ and $v \neq b$.  Suppose further that it is false that $ (a,v) $ and $ (v,b) $ are both arcs in $N$.  Then $D(v) D(a,b) N = D(a,b) D(v) N$.  \\
(iii) Assume $v \neq a$ and $v \neq b$.  Assume that $ (a,v) $ and $ (v,b) $ are arcs in $N$.  Then $D(v) D(a,b) N = D(v) N$.  \\
 (iv) If $v = a$ , write $q_1, \cdots , q_k$ for the parents of $a$.  Then 
$$D(a) D(a,b) N = \big[\prod_{i=1}^k D(q_i,b)\big] D(a) N.$$   
(v) If $v = b$, write $c_1, \cdots , c_k$ for the children of $b$.   Then
$$D(b) D(a,b) N = \big[\prod_{j=1}^k D(a,c_j)\big] D(b) N.$$  
\end{thm}

\begin{proof}

(i) Case 1.  Suppose neither $v$ nor $w$ is a parent of the other.  
Let $q_1, \cdots , q_k$ be the parents of $v$ and $c_1, \cdots , c_n$  the children of $v$.
Let $r_1, \cdots , r_l$ be the parents of $w$ and $d_1, \cdots , d_m$  the children of $w$.
By the assumptions, $w$ is not any $q_i$ nor $c_i$; and $v$ is not any $r_i$ or $d_i$.

Then $D(v) N$ has new arcs $ (q_i,c_j) $.  Neither of the vertices of such an arc is $w$.  Now delete $w$.  In $D(w) D(v) N$ we get new arcs $ (r_i, d_j) $ since $v$ is not any $r_i$ nor $d_i$.  

A similar argument shows that we obtain the same arcs in $D(v) D(w) N$.  

Case 2.  Suppose one vertex is a parent of the other.  Without loss of generality, assume $v$ is a parent of $w$.  

Let $q_1, \cdots , q_k$ be the parents of $v$ and $c_1, \cdots , c_n, w$  the children of $v$.
Let $v, r_1, \cdots , r_l$ be the parents of $w$ and $d_1, \cdots , d_m$ the children of $w$.

First we remove $v$.  The new arcs  of $D(v) N$ are $ (q_i,c_j) $ and  $ (q_i,w) $.
Next we remove $w$.  The arcs $ (q_i,w)$, $(r_i,w)$, and $(w, d_i) $ are removed while the arcs $ (q_i,c_j) $ remain.  The new arcs in $D(w) D(v) N$ are thus $(r_i,d_j)$ and $ (q_i,c_j) $. 

Alternatively we first  remove $w$ from $N$.  The new arcs of $D(w) N$ are $ (r_i,d_j) $ and $ (v,d_j) $.  Next we remove $v$.   The arcs $ (v,d_j) $ are removed while the arcs $ (r_i,d_j) $ remain.  The new arcs of $D(v) D(w) N$ are thus $ (q_i,c_j) $ and $ (r_i,d_j) $.

The vertices and arcs of $D(v) D(w) N$ and $D(w) D(v) N$ are the same, so the conclusion follows.  

(ii) 
Let $v$ have parents $q_1, q_2, \cdots , q_k$  and children $c_1, \cdots , c_m$. 
$D(a,b) N$ has all the arcs of $N$ except $ (a,b) $. 
Then $D(v) D(a,b) N$ removes $v$ and all arcs involving $v$ and adds new arcs $ (q_i,c_j) $.  Note that no arc $ (q_i,c_j) $ is the same as $ (a,b) $ unless $a=q_i$ for some $i$ and $b = c_j$ for some $j$, (in which case $ (a,b) $ is present again in $D(v) D(a,b)N$).  But this case is excluded by hypothesis.  

Similarly, $D(v) N$ removes $v$ and all arcs involving $v$.  It adds new arcs $ (q_i, c_j) $  (if not already present).  The arc $ (a,b) $ is still present.  It remains redundant unless the only path in $N$ from $a$ to $b$ that avoids $ (a,b) $ was $a,v,b$ in which case $a$ is a parent of $v$ and $b$ is a child of $v$; this case is excluded by hypothesis.  
Next $D(a,b) D(v) N$  deletes $ (a,b) $.

Thus $D(a,b) D(v) N = D(v) D(a,b) N$
since their vertices and arcs are the same.  

(iii) 
Let $v$ have parents $a, q_1, q_2, \cdots , q_k$  and children $b, c_1, \cdots , c_m$. 
Then $D(a,b) N$ has all the arcs of $N$ except $ (a,b) $. 
Next $D(v) D(a,b) N$ removes $v$ and all arcs involving $v$ and adds new arcs $ (a,b) $, $ (a, c_j) $, $ (q_i,b) $ and $ (q_i,c_j) $.  Note that $ (a,b) $ is still present in $D(v) D(a,b)N$.  

Similarly, $D(v) N$ removes $v$ and all arcs involving $v$.  It adds new arcs $ (a,c_j) $, $ (q_i,b) $, $ (q_i, c_j) $  (if not already present).  The arc $ (a,b) $ is still present. 
We see that $ D(v) N = D(v) D(a,b) N$
since their vertices and arcs are the same.

(iv) 
Let the children of $v = a$ in $N$ be $c_1, \cdots , c_m, b$.
Then $D(a,b) N$ has the same arcs as $N$ except that $ (a,b) $ is missing.  Hence
$D(a) D(a,b) N$ has all arcs of $N$ removed that involve $a$, and it has new arcs
$ (q_i,c_j)$.

On the other hand $D(a) N$ has all arcs involving $a$ removed and new arcs 
$ (q_i,c_j) $ and $ (q_i,b)$.

Note that in $N$, $a$ has a child $d$ so $d < b$ since $ (a,b) $ was redundant in $N$.  Hence $ (q_i,b) $ is redundant since there is a path $q_i < d < b$.
Thus $D(q_i,b) $ is a well-defined operation that deletes $(q_i,b)$.

It follows that  $\big[ \prod_{i=1}^k D(q_i,b)\big] D(a) N$ has the same arcs as $N$ except that all arcs involving $a$ have been removed, and there are new arcs $ (q_i,c_j) $.
The network is thus the same as $D(a) D(a,b) N$.

(v)
Suppose that in $N$, the parents of $b$ are $a,  q_1, \cdots , q_m$.  Then $D(a,b) N$ is the same as $N$ except that $ (a,b) $ is removed.   The parents of $b$ are now only  $q_1, \cdots , q_k$.  Then $D(b) D(a,b) N$ has all arcs involving $b$ removed, plus new arcs $ (q_i,c_j) $.

On the other hand, $D(b) N$ has all the arcs of $N$ involving $b$ removed and has new arcs $ (q_i,c_j) $ and $ (a, c_j) $. 
Note that $ (a,c_j) $ is redundant in $D(b) N$ since in $N$ there is a child $d$ of $a$ and a path $a<d<b<c_j$.  If we apply $D(a,c_j) $ then all that changes is that $ (a,c_j) $ is removed.
Thus  $\big[\prod_{j=1}^k D(a,c_j)\big] D(b) N$ has all the arcs of $N$ involving $b$ removed and with new arcs $ (q_i,c_j) $.  The network thus agrees with $D(b) D(a,b) N$.
\end{proof}

\begin{cor} 
Let $N'$ be a CPS of $N$ with vertex set $V(N)-W$, where $W$ is a subset of $V(N) $.  Let $W = \{w_1, w_2, \cdots , w_k\}$ (in any order).  Then $N'$ can be written in the form
$R D(w_k) \cdots  D(w_1) N$
where $R$ is a product of $D(a_i,b_i) $ for redundant arcs $ (a_i, b_i) $ of $D(w_k) \cdots  D(w_1) N$.
\end{cor}

\begin{proof}
 By hypothesis, $N'$ can be written as a composition $D_m \cdots  D_1 N$ where each $D_i$ is of form either $D(w_j) $ or $D(a,b) $.  But by Theorem 6.1, each $D(w_j) D(a,b) $ can be rewritten as $S D(w_j) $ where $S$ is a product of $D(a_i,b_i) $.  In this way all the $D(w_j) $ can be moved to the right so that there are no deletions of a redundant arc to the right of any $D(w_j) $.  
\end{proof}

\section { Invariant properties} 

This section proves the following  partial converse to Theorem 5.2 and considers its consequences.  

\begin{thm} 
Let $N$ and $N'$ be DC $X$-networks such that  $V(N') \subseteq V(N) $.  Assume\\
(i) Whenever $u < v$ in $N'$, then $u < v$ in $N$.\\
(ii) Whenever $u$ and $v$ are in $V(N') $ and $u < v$ in $N$, then $u < v$ in $N'$.\\
Suppose $N'$ contains no redundant arc.  
Then $N'$ is a CPS of $N$.
\end{thm}

Thus conditions (i) and (ii) together with the absence of redundant arcs imply that $N'$ is a CPS, obtained from $N$ by a composition of operations of form $D(v) $ and $D(a,b) $.  

Since the networks are DC, the condition $V(N') \subseteq V(N) $ can be rephrased by saying that there exists $\phi: V(N') \to V(N) $ such that, for all $v\in V(N')$, $cl(v; N') = cl(\phi(v); N) $.

\begin{proof}
 Let $W = V(N) - V(N') $.  Write $W = \{w_1, w_2, \cdots , w_k\}$ in some fixed order.  There is a composition $S = D(a_m,b_m) \cdots   D(a_1,b_1) $ of deletions of redundant arcs $ (a_i,b_i) $ of $N$ such that
$N_1 := S N$ contains no redundant arcs.  Let 
$N_2 = D(w_k) \cdots  D(w_1) N_1$.
There is then a composition $U$ of deletion of redundant arcs such that
$N_3 = U N_2$ contains no redundant arcs.
Note that $N_3$ is a CPS of $N$ by definition.

We claim that $N' = N_3$.   By the choice of $W$, $N'$ and $N_3$ have the same vertex set.  There remains to show that the arcs of $N'$ and $N_3$ are the same.

Let $ (u,v) $ be an arc of $N'$.   By (i) $u < v$ in $N$.   Since $N_3$ is a CPS of $N$, it follows $u < v$ in $N_3$ by Theorem 5.2.  If $ (u,v) $ is not an arc of $N_3$, then there exists a vertex $e$ of $N_3$ such that $u < e < v$ in $N_3$.  Note $e$ is a vertex of $N'$ as well by the choice of $W$.  Since $N_3$ is a CPS it follows $u < e < v$ in $N$, whence by (ii) $u < e < v$ in $N'$.  It follows that $ (u,v) $ is redundant in $N'$, a contradiction.  This shows that $ (u,v) $ is an arc of $N_3$.  

Now assume that $ (u,v) $ is an arc of $N_3$.  Since $N_3$ is a CPS of $N$, it follows that $u < v$ in $N$.   By (ii), it follows that $u < v$ in $N'$.  We claim $ (u,v) $ is an arc of $N'$.  If not, then there exists a vertex $e$ of $N'$ such that $u < e < v$ in $N'$.  By (i), $u < e < v$ in $N$.  Since $N_3$ is a CPS of $N$ it follows that $u < e < v$ in $N_3$.  Hence $ (u,v) $ is redundant in $N_3$, a contradiction.   This shows that $ (u,v) $ is an arc of $N'$.
\end{proof}

\begin{lem} 
Suppose $N$ and $N'$ are DC $X$-networks with the same vertex set $V$.  Assume\\
(i) Whenever $u < v$ in $N'$ then $u < v$ in $N$.\\
(ii) Whenever $u < v$ in $N$ then $u < v$ in $N'$.\\
Suppose neither $N$ nor $N'$ has a redundant arc.  Then $N = N'$.
\end{lem}

\begin{proof}
The vertex sets are the same.  We show that the arcs are the same.

Suppose $ (a,b) $ is an arc of $N'$.  Then $a < b$ in $N'$ so $a < b$ in $N$.  We claim that $ (a,b) $ is an arc of $N$.  If not then there exists $c$ such that $a < c < b$ in $N$.  It follows that $a < c < b$ in $N'$ as well, whence $ (a,b) $ is a redundant arc of $N'$, a contradiction.  Hence $ (a,b) $ is an arc of $N$.

A symmetric argument shows that if $ (a,b) $ is an arc of $N$, then $ (a,b) $ is an arc of $N'$.
It follows that $N = N'$ since they have the same vertices and arcs.  
\end{proof}

Write $X = \{x_1, x_2, \cdots , x_n\}$.  Let $Tr(X)$ denote the set of trivial clusters, $Tr(X) = \{X, \{x_1\}, \{x_2\}, \cdots , \{x_n\}\}$.

\begin{thm} 
Let $N = (V,E,r,X) $ be a DC $X$-network.  Let $V'$ be a subset of $V$ containing $Tr(X)$.  There exists a unique DC $X$-network $N'$ such that\\
(1) $V(N')  = V'$.\\
(2) $N'$ has no redundant arcs.\\
(3) $N'$ is a CPS of $N$.
\end{thm}

\begin{proof}
Let $U = V - V'$.  The members of $U$ are the vertices that must be deleted from $N$ to obtain the vertices of $N'$.  Write $U = \{u_1, u_2, \cdots , u_k\}$ in some fixed order.  There is a composition $S_1 = D(a_m,b_m) \cdots   D(a_1,b_1) $ of deletion of redundant arcs  such that
$N_1 := S_1 N$ contains no redundant arcs.  Let 
$N_2 := D(u_k) \cdots  D(u_1) N_1$.
There is then a composition $S_2$ of deletions of redundant arcs such that
$N_3 = S_2 N_2$ contains no redundant arcs.
Note that $N_3$ is a CPS of $N$ by definition and $V(N_3) = V'$.  This proves the existence.

For the uniqueness, let $N'$ be another such $X$-network.  If $u < v$ in $N_3$, then $u < v$ in $N$, so $u < v$ in $N'$.   Similarly if $u < v$ in $N'$, then $u < v$ in $N$, so $u < v$ in $N_3$.
Since neither $N_3$ nor $N'$ has redundant arcs, Lemma 7.2 implies $N' = N_3$.
\end{proof}

The unique $N'$ of Theorem 7.3 will be denoted $N(W) $, where $W = V(N') - Tr$.
Note that the members of $W$ are the nontrivial clusters of $N'$. Thus $N(W) $ is the unique CPS of $N$ with no redundant arcs and with vertex set $W \cup Tr$.  We may also denote it by $N(w_1, \cdots , w_k) $ where the elements of $W$ are $w_1, \cdots , w_k$.  

\begin{thm} 
Let $N = (V,E,r,X) $ be a DC $X$-network.  Let $W$ and $W'$ be subsets of $V$ disjoint from $Tr$.  Then $N(W') $ is a CPS of $N(W) $ iff $W' \subseteq W$.
\end{thm}

\begin{proof}
 If $N(W') $ is a CPS of $N(W) $, then \\
$$Tr \cup W' = V(N(W')) \subseteq V(N(W)) = W \cup Tr.$$   
Since both $W$ and $W'$ are disjoint from $Tr$, it follows $W' \subseteq W$.

Conversely, suppose $W' \subseteq W$.  Let  $v_1, \cdots , v_k$ be a listing of $V-W'$ such that 
$v_1, \cdots , v_j$ are the members of $V - W$ while $v_{j+1}, \cdots , v_k$ are the members of $W - W'$.  
There exists a composition $R$ of deletions of redundant arcs such that 
$N(W) = R D(v_j) \cdots  D(v_1) N$.
Let $N_1 = D(v_k) \cdots  D(v_{j+1}) N(W) $.   This network has vertex set $Tr \cup W'$.  There is a composition $S$ of deletions of redundant arcs such that $S N_1$ has no redundant arcs.   By Theorem 7.3, $S N_1 = N(W') $.  Hence $N(W') $ is a CPS of $N(W) $.
\end{proof}

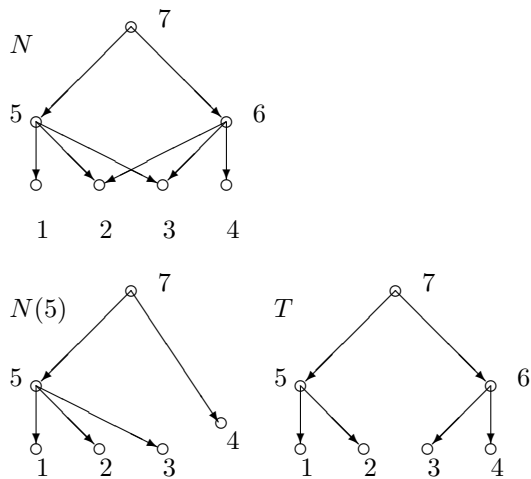
\begin{figure}[!htb]  
\begin{center}

\begin{picture}(200,190) (0,0)  

\put(10,134){\vector(0,-1){22}}
\put(10,134){\vector(1,-1){22}}
\put(10,134){\vector(2,-1){46}}
\put(10,134){\circle{4}}
\put(82,134){\vector(-2,-1){46}}
\put(82,134){\vector(0,-1){22}}
\put(82,134){\vector(-1,-1){22}}
\put(82,134){\circle{4}}
\put(46,170){\vector(-1,-1){34}}
\put(46,170){\vector(1,-1){34}}
\put(46,170){\circle{4}}
\put(10,110){\circle{4}}
\put(34,110){\circle{4}}
\put(58,110){\circle{4}}
\put(82,110){\circle{4}}

\put(10,90){1}
\put(34,90){2}
\put(58,90){3}
\put(82,90){4}
\put(92,134){6}
\put(0,134){5}
\put(56,170){7}
\put(0,160){$N$}

\put(10,34){\vector(0,-1){22}}
\put(10,34){\vector(1,-1){22}}
\put(10,34){\vector(2,-1){46}}
\put(10,34){\circle{4}}
\put(46,70){\vector(-1,-1){34}}
\put(46,70){\vector(2,-3){33}}
\put(46,70){\circle{4}}
\put(10,10){\circle{4}}
\put(34,10){\circle{4}}
\put(58,10){\circle{4}}
\put(80,20){\circle{4}}

\put(10,0){1}
\put(34,0){2}
\put(58,0){3}
\put(82,10){4}
\put(0,34){5}
\put(56,70){7}
\put(0,60){$N(5)$}

\put(110,34){\vector(0,-1){22}}
\put(110,34){\vector(1,-1){22}}
\put(110,34){\circle{4}}
\put(182,34){\vector(0,-1){22}}
\put(182,34){\vector(-1,-1){22}}
\put(182,34){\circle{4}}
\put(146,70){\vector(-1,-1){34}}
\put(146,70){\vector(1,-1){34}}
\put(146,70){\circle{4}}
\put(110,10){\circle{4}}
\put(134,10){\circle{4}}
\put(158,10){\circle{4}}
\put(182,10){\circle{4}}

\put(110,0){1}
\put(134,0){2}
\put(158,0){3}
\put(182,0){4}
\put(192,34){6}
\put(100,34){5}
\put(156,70){7}
\put(100,60){$T$}

\end{picture}

\caption{ A DC $X$-network $N$  with tree $N(5)$ that is a CPS of $N$ and also a tree $T$ that is displayed by $N$ but is not a CPS of $N$.  
 }  
\end{center}
\end{figure}

An $X$-network $N'$ is \emph{displayed} by an $X$-network $N$ if $N'$ is obtained by deleting arcs of $N$ and possibly contracting arcs $(a,b)$ if $a$ has out-degree one.   It is important to recognize that, for example, having a tree $T$ displayed by a DC network $N$ is not the same as having the tree $T$ be a CPS of $N$.  Figure 3 shows a network $N$ and a CPS $N(5)$ of $N$ which is displayed by $N$.  It also shows a tree $T$ that is displayed by $N$ but that is not a CPS.  One recognizes that $T$ is not a CPS since $cl(5;T) =  \{1,2\}$, which is not the cluster of any vertex of $N$.  

Suppose $N$ has no redundant arcs.  Suppose $N'$ satisfies the hypotheses of Theorem 7.1 except that $N'$ contains a redundant arc.  Then $N'$ need not be a CPS of $N$.  To see this, consider the network $N$ in Figure 4, which is a rooted tree.  The accompanying network $N'$ satisfies the hypotheses of Theorem 7.1 except that there is a redundant arc.  But every CPS of a tree is a tree by Theorem $5.3$, so $N'$ is not a CPS of $N$.  Hence Theorem 7.1 is not true without the assumption that there are no redundant arcs.

\begin{figure}[!htb]  
\begin{center}

\begin{picture}(200,115) (0,0)
\put(40,40){\vector(-1,-1){28}}
\put(40,40){\vector(1,-1){28}}
\put(40,40){\circle{4}}
\put(40,70){\vector(0,-1){28}}
\put(40,70){\vector(1,-1){28}}
\put(40,70){\circle{4}}
\put(40,100){\vector(0,-1){28}}
\put(40,100){\vector(1,-1){28}}
\put(40,100){\circle{4}}
\put(10,10){\circle{4}}
\put(70,10){\circle{4}}
\put(70,40){\circle{4}}
\put(70,70){\circle{4}}

\put(0,10) {1}
\put(55,10) {2}
\put(75,40){3}
\put(75,70){4}
\put(25,40){5}
\put(25,70){6}
\put(25,100){7}
\put(0,100){$N$}

\put(140,40){\vector(-1,-1){28}}
\put(140,40){\vector(1,-1){28}}
\put(140,40){\circle{4}}
\put(140,100){\vector(0,-1){58}}
\put(140,100){\vector(1,-2){29}}
\put(140,100){\vector(1,-1){28}}
\put(140,100){\vector(-1,-3){29}}
\put(140,100){\circle{4}}
\put(110,10){\circle{4}}
\put(170,10){\circle{4}}
\put(170,40){\circle{4}}
\put(170,70){\circle{4}}

\put(100,10) {1}
\put(155,10) {2}
\put(175,40){3}
\put(175,70){4}
\put(125,40){5}
\put(25,70){6}
\put(125,100){7}
\put(100,100){$N'$}

\end{picture}

\caption{ $N'$ is not a CPS of $N$   
 }  
\end{center}
\end{figure}
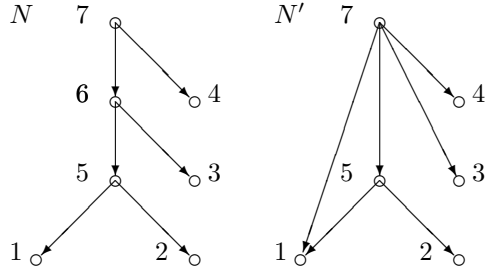

The next result shows that the example is typical.  

\begin{thm} 
Let $N$ and $N'$ be DC $X$-networks such that $V(N') \subseteq V(N) $.  Assume\\
(i) Whenever $u < v$ in $N'$, then $u < v$ in $N$.\\
(ii) Whenever $u$ and $v$ are in $V(N')$ and $u<v$ in $N$, then $u < v$ in $N'$.\\
Then there exists a CPS $N'' = (V'', E'', r, X)$ of $N$ with no redundant arcs and a collection of pairs $(a_i,b_i)$ for $i = 1, \cdots , k$ where $a_i$ and $b_i$  lie in $V''$, $a_i < b_i$ in $N''$,  $(a_i,b_i)\notin E''$, such that
$N'$  is obtained by adjoining arcs $(a_i,b_i)$ to $E''$.  
\end{thm}

\begin{proof}
  
Let $ (a_1, b_1), \cdots,  (a_k,b_k) $ denote the redundant arcs of $N'$.\\
Let
$N'' = D(a_k,b_k) \cdots  D(a_1,b_1) N'$.  Then $N''$ has no redundant arcs.   By Theorem 5.2, $N''$  satisfies \\
(i) Whenever $u < v$ in $N''$, then $u < v$ in $N$.\\
(ii) Whenever $u$ and $v$ are in $V(N'') $ and $u < v$ in $N$, then $u < v$ in $N''$.

Hence by Theorem 7.1, $N''$ is a CPS of $N$.  It follows that $N'$ is obtained by adjoining the arcs $ (a_i,b_i) $ back to the CPS $N''$.
\end{proof}

\section { An example} 

In this section we compute the inheritance metric in an example and also find some cluster-preserving simplifications of an initial network $N\in DC(X)$.  For this example, we are also able to compute  CPS trees that ``best fit''  the initial network $N$.  The idea is that an initial  $N \in DC(X) $  could be very complicated and the following problem then arises:

\textbf{Problem}.  Given $N \in DC(X) $, let $\mathcal{T}(N)$ be the collection of CPS of $N$ which are trees.  Find $T_0 \in \mathcal{T}(N)$ such that $T_0$ minimizes $D(N,T)$ for $T\in\mathcal{T}(N)$.

We call a solution to the problem a \emph{best fitting CPS tree for N}.  

More generally, given a subset $\mathcal{S} \subset DC(X) $, we might seek a member $S_0 \in \mathcal{S}$ such that $S_0$ minimizes $D(N,S)$ for $S \in\mathcal{S}$.

Here we begin the example. Let $N$ be the network in Figure 1.  Note that $X = \{1,2,3,4\}$ and  $N$ is DC.  List the 9 vertices in their order 1, 2, $\cdots$,  9.  Then the adjacency matrix $A$ is the matrix

$$A=\left[\begin{array}{ccccccccc}
0	&0	&0	&0	&0	&0	&0	&0	&0\\
0	&0	&0	&0	&0	&0	&0	&0	&0\\
0	&0	&0	&0	&0	&0	&0	&0	&0\\
0	&0	&0	&0	&0	&0	&0	&0	&0\\
1	&1	&0	&0	&0	&0	&0	&0	&0\\
0	&1	&1	&0	&0	&0	&0	&0	&0\\
0	&0	&1	&1	&0	&0	&0	&0	&0\\
0	&0	&0	&0	&1	&1	&0	&0	&0\\
0	&0	&1	&0	&0	&0	&1	&1	&0
\end{array} \right]$$

The longest path has length 3, so $H = I + A + A^2 + A^3$.  Hence the inheritance matrix $H = H(N) $ is 

$$H = \left[\begin{array}{ccccccccc}
1	&0	&0	&0	&0	&0	&0	&0	&0\\
0	&1	&0	&0	&0	&0	&0	&0	&0\\
0	&0	&1	&0	&0	&0	&0	&0	&0\\
0	&0	&0	&1	&0	&0	&0	&0	&0\\
1	&1	&0	&0	&1	&0	&0	&0	&0\\
0	&1	&1	&0	&0	&1	&0	&0	&0\\
0	&0	&1	&1	&0	&0	&1	&0	&0\\
1	&2	&1	&0	&1	&1	&0	&1	&0\\
1	&2	&3	&1	&1	&1	&1	&1	&1
\end{array} \right]$$

As a check, note that $H_{9,3} = 3$ since there are 3 directed paths from 9 to 3.  

Suppose one sought examples of CPS of $N$ that are trees.  Consider, for example, $T_1 = 
D(9,3) D(7) D(8,2) D(6) D(9,3) N$, shown in Figure 5. 
Vertices have the same labels in $T_1$ as in $N$. This is possible since the vertices are identified with their clusters.  Since $T_1$ has no redundant arcs and nontrivial vertices 5 and 8, $T_1 = N(5,8) $ by Theorem 7.3.  

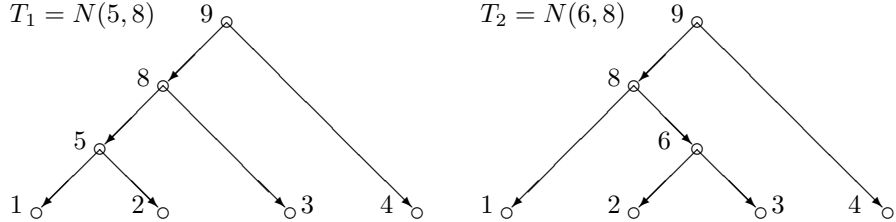
\begin{figure}[!htb]  
\begin{center}

\begin{picture}(350,100) (0,0)
\put(34,34){\vector(-1,-1){22}}
\put(34,34){\vector(1,-1){22}}
\put(34,34){\circle{4}}
\put(58,58){\vector(-1,-1){22}}
\put(58,58){\vector(1,-1){46}}
\put(58,58){\circle{4}}
\put(82,82){\vector(-1,-1){22}}
\put(82,82){\vector(1,-1){70}}
\put(82,82){\circle{4}}
\put(10,10){\circle{4}}
\put(58,10){\circle{4}}
\put(106,10){\circle{4}}
\put(154,10){\circle{4}}

\put(0,10) {1}
\put(46,10) {2}
\put(110,10){3}
\put(140,10){4}
\put(24,34){5}
\put(48,58){8}
\put(72,82){9}
\put(0,82){$T_1=N(5,8)$}

\put(236,58){\vector(-1,-1){46}}
\put(236,58){\vector(1,-1){22}}
\put(236,58){\circle{4}}

\put(260,34){\vector(-1,-1){22}}
\put(260,34){\vector(1,-1){22}}
\put(260,34){\circle{4}}
\put(260,82){\vector(-1,-1){22}}
\put(260,82){\vector(1,-1){70}}
\put(260,82){\circle{4}}
\put(188,10){\circle{4}}
\put(236,10){\circle{4}}
\put(284,10){\circle{4}}
\put(332,10){\circle{4}}

\put(178,10) {1}
\put(226,10) {2}
\put(288,10){3}
\put(317,10){4}
\put(245,34){6}
\put(226,58){8}
\put(250,82){9}
\put(178,82){$T_2=N(6,8)$}

\end{picture}

\caption{ The best fitting CPS trees $T_1=N(5,8)$ and $T_2=N(6,8)$ for $N$ in Figure 1.   
 }  
\end{center}
\end{figure}

To compute the distance $D(N,T_1) $, let $C$ consist of all clusters of $N$, identified with the vertices of $N$.  Note that $T_1$ does not contain 6 or 7.  Then 

$$_CH(T_1)=\left[\begin{array}{ccccccccc}
1	&0	&0	&0	&0	&0	&0	&0	&0\\
0	&1	&0	&0	&0	&0	&0	&0	&0\\
0	&0	&1	&0	&0	&0	&0	&0	&0\\
0	&0	&0	&1	&0	&0	&0	&0	&0\\
1	&1	&0	&0	&1	&0	&0	&0	&0\\
0	&0	&0	&0	&0	&0	&0	&0	&0\\
0	&0	&0	&0	&0	&0	&0	&0	&0\\
1	&1	&1	&0	&1	&0	&0	&1	&0\\
1	&1	&1	&1	&1	&0	&0	&1	&1
\end{array} \right]$$

From $_CH(N) = H(N) $ and $_CH(T_1) $, it is easy to compute $D(N,T_1) = 13$.

Recall that $Tr(X)$ denotes the set of trivial clusters.  In our example, when we identify the clusters with their corresponding vertices, $Tr(X) = \{1,2,3,4,9\}$.  

Suppose we seek the best fitting CPS trees of $N$ by exhaustive search.  Since there are $n = 4$ leaves, there are 15 binary rooted  trees and 11 non-binary rooted trees, hence a total of 26 possible rooted trees.  Any rooted tree contains no redundant arc.  By Theorem 7.3 if it is a CPS, it must have the form $N(S) $ for a subset $S$ of $\{5,6,7,8\}$; hence there are at most $2^4 = 16$ networks to check.  Moreover, since a rooted tree with 4 leaves can have at most 7 vertices, in our search $S$ may contain at most 2 vertices not in $Tr(X)$; this eliminates 5 of the 16 networks.  Table 1 shows the 11 possibilities and their distances from $N$.

\begin{table}[!htbp]   
\begin{center}
\begin{tabular}{l|ll}
CPS	$N'$	&$D(N,N') $	\\
\hline
$N(\emptyset) $	&23	\\
$N(5) $		&19	\\
$N(6) $		&19	\\
$N(7) $		&19	\\
$N(8) $		&18	\\
$N(5,6) $	  	&14  		&not a tree	\\
$N(5,7) $		&15	\\
$N(5,8) $		&13	 \\
$N(6,7) $	 	 &14 &not a tree \\ 
$N(6,8) $	 	&13	  \\
$N(7,8) $		 &13  	&not a tree \\ 
\end{tabular}
\end{center}
\caption{The 11 CPS of $N$ with two or fewer nontrivial vertices.}
\end{table}

Exhaustive search in this case shows that there is a tie for the best fitting CPS tree of $N$ between $T_1=N(5,8) $ and $T_2=N(6,8) $, both shown in Figure 5.  Both have distance 13 
from $N$.   Note that of the 26 rooted trees, only 8 are CPS of N.  For example, the tree $T_3$
 with nontrivial clusters $\{3,4\}$ and $\{2,3,4\}$ 
 is not a CPS of $N$; this is obvious since $\{2,3,4\}$ is a cluster of $T_3$ but not of $N$.  Since $T_3$ is DC we may nevertheless compute $D(N,T_3) = 25$.  By Theorem 7.5, $N(6) $ is a CPS of $N(6,7) $ but not of $N(5,7,8) $.   

\section{Extensions} 

\textbf{(I) Hybrid vertices with out-degree 1}

Some papers (such as \cite{nak05}, \cite{mor04}, \cite{crv09}) require that in a network  every hybrid vertex $v$ has out-degree 1.  If $c$ is the unique child of $v$, then $cl(v)=cl(c)$, and the network cannot be distinct-cluster.  The results in this paper may nevertheless apply to such networks as follows.

A \emph{hybrid out-degree-1} $X$-network $N$ is $N=(V,E,r,X)$ where $(V,E)$ is a finite acyclic directed network with root $r$ and leaf set $X$ which satisfies that every hybrid vertex $v$ has out-degree 1.   Let $H_{o1}(X)$ denote the collection of hybrid out-degree-1 $X$-networks.  Given $N\in H_{o1}(X)$, form a new network as the result of contracting each edge between a hybrid vertex and its unique child.  More specifically, given a hybrid vertex $v$ with parents $q_1, q_2, \cdots, q_k$ and a unique child $c$ let $V'=V-\{v\}$ and $E'=[E-\{(q_1,v),\cdots, (q_k,v), (v,c)\}] \cup \{(q_1,c),\cdots, (q_k,c)\}$.   Thus the hybrid vertex has been replaced by its former child, which is now hybrid with in-degree $k$.  No member of $X$ is deleted, so that  $(V',E')$ is easily seen to be a finite acyclic directed network with root $r$ and leaf set $X$.  It is possible that a hybrid vertex is now a leaf.  If this procedure is used recursively until there are no more hybrid vertices with out-degree 1, we obtain a network we shall call the \emph{derivative of N with non-unit hybrid out-degree } and denote $H_{o\neq1}(N)$.   The order of the various deletions do not affect the resulting network.  

Conversely, suppose  $N=(V,E,r,X)$ satisfies that $(V,E)$ is a finite acyclic directed network with root $r$ and leaf set $X$ but no hybrid vertex has out-degree 1.  We may construct a new network
in which every hybrid vertex has out-degree 1 by reversing the previous procedure.  More specifically, if $v$ is hybrid with parents $q_1, \cdots,q_k$ and out-degree either 0 or greater than 1, we insert a new vertex $w$, remove the arcs $(q_1,v), \cdots, (q_k,v)$  and add new arcs $(w,v)$ and $(q_1,w),\cdots,(q_k,w)$.  Note that now $w$ is hybrid with out-degree 1 because it has a unique child $v$.  If $v\in X$ then $v$ remains in $X$ after the construction.  We perform this procedure until all hybrid vertices have out-degree 1.  Call the result the \emph{derivative of N with unit hybrid out-degree} and denote it $H_{o1}(N)$.  

It is straightforward to see that if $N \in H_{o1}(X)$, then $H_{o1}(H_{o\neq1}(N))$ is isomorphic with $N$.  The argument uses that there is a one-to-one correspondence between hybrid vertices of $N$ and hybrid vertices of $H_{o\neq1}(N)$, while every hybrid vertex of $N$ has out-degree 1. 

A network $N \in H_{o1}(X)$ is \emph{extended distinct-cluster} if $H_{o\neq1}(N)$ is distinct-cluster.  Roughly, $N$ is extended distinct-cluster when distinct vertices have different clusters, except that a hybrid and its unique child are allowed to have the same cluster.  Let $DC_{o1}(X)$ denote the set of extended distinct-cluster networks.  For $N_1$ and $N_2$ in $DC_{o1}(X)$ define
$$D_{o1}(N_1,N_2) = D(H_{o\neq1}(N_1), H_{o\neq1}(N_2)).$$
The definition makes sense since $H_{o\neq1}(N_1)$ and  $H_{o\neq1}(N_2)$ are in $DC(X)$.  

\begin{thm}
$D_{o1}$ is a metric on $DC_{o1}(X)$.
\end{thm}

\begin{proof}
Suppose $D_{o1}(N_1,N_2)=0.$  By definition $D(H_{o\neq1}(N_1), H_{o\neq1}(N_2))=0$, so that $H_{o\neq1}(N_1) = H_{o\neq1}(N_2)$ because $D$ is a metric on $DC(X).$   But then
$N_1 = H_{o1}(H_{o\neq1}(N_1)) = H_{o1}(H_{o\neq1}(N_2)) = N_2$.  The other properties of a metric are immediate.  
\end{proof}

\textbf{(II) Use of $p$-norms}

The construction of the metric $D$ in section 4 can be easily generalized.  For $p \geq 1$ let $|| . ||_p$ denote the $p$-norm for an $m\times m$ matrix $M$ regarded as being in $\mathbb{R}^{m^2}$. Thus 
$$||M||_p = \Big[\sum [|M_{u,v}|^p: 1 \leq u, v \leq m] \Big]^{1/p}.$$  

Let $N_1$ and $N_2$ be in $DC(X)$.  Let $C$ be a set of $X$-clusters containing all clusters of $N$ and $N'$.  Our definition of the inheritance distance between $N$ and $N'$ is equivalent to
$_CD(N,N') := ||_CH - \,_CH' ||_1$ and $D(N,N') = \,_CD(N,N')$ for any such $C$.
By analogy we may define 
$_CD_p(N,N') := ||\,_CH - \,_CH' ||_p$ and then define the \emph{$p$-norm inheritance distance} between $N$ and $N'$ by
$D_p(N,N') = \,_CD_p(N,N')$ for any such $C$.  By an argument like that of Lemma 4.1, the result is the same  for any $C$ containing the clusters of both $N$ and $N'$. Note that $D_p(N,N') $ will be the $p$-th root of an integer.    In the example of Section 8, $D_2(N,N(5,8))= \sqrt{15}$.  

\textbf{Acknowledgments.}  We wish to thank the anonymous referees whose recommendations provided substantial improvements to this paper.

\end{document}